\documentclass[12pt]{amsart}
\usepackage{latexsym,amsxtra,amscd,ifthen,amsmath,color, multicol,hyperref,mathdots,mathrsfs}
\usepackage{amsfonts}
\usepackage{verbatim}
\usepackage{amsmath}
\usepackage{amsthm}
\usepackage{amssymb}
\usepackage[all,cmtip]{xy}
\usepackage[normalem]{ulem}
\usepackage{enumerate}
\usepackage[latin1]{inputenc}
\usepackage{graphicx}
\usepackage{multirow}
\usepackage{latexsym}
\swapnumbers
\theoremstyle{plain}
\newtheorem{thm}{Theorem}[section]
\newtheorem{lem}[thm]{Lemma}

\newtheorem{cor}[thm]{Corollary}

\theoremstyle{definition}

\newtheorem*{Ack}{Acknowledgement}
\newtheorem*{thm*}{Theorem}

\theoremstyle{remark}
\newtheorem{note}[thm]{}

\def\Z{\mathbb{Z}}

\def\dim{\operatorname{dim}}

\begin{document}

\title{A note on generic Clifford algebras of binary cubic forms}

\author{Linhong Wang}
\email{lhwang@pitt.edu}
\address{Department of Mathematics\\
University of Pittsburgh, Pittsburgh, PA 15260}

\author{Xingting Wang*}
\email{xingting.wang@howard.edu}
\address{Department of Mathematics\\Howard University, Washington, D.C. 20059}

\begin{abstract}
We study the representation theoretic results of the binary cubic generic Clifford algebra $\mathcal C$, which is an Artin-Schelter regular algebra of global dimension five. In particular, we show that $\mathcal C$ is a PI algebra of PI degree three and compute its point variety and discriminant ideals. As a consequence, we give a necessary and sufficient condition on a binary cubic form $f$ for the associated Clifford algebra $\mathcal C_f$ to be an Azumaya algebra. 
\end{abstract}

\thanks{*corresponding author}

\maketitle

\section{Introduction}
Let $f$ be a form of degree $m$ in $n$ variables over a base field $k$. The Clifford algebra $\mathcal C_f$ associated to the form $f$ is defined to be an associative algebra $k\left\langle x_1,\dots,x_n\right\rangle$ subject to the relations $(a_1x_1+\cdots+a_nx_n)^m-f(a_1,\dots,a_n)$ for all $a_1,\dots,a_n\in k$. In \cite{CYZ}, Chan-Young-Zhang defined the generic Clifford algebra (or more precisely universal Clifford algebra) $\mathcal C_{m, n}=k\left\langle x_1,\dots,x_n\right\rangle$ subject to the relations 
$(a_1x_1+\cdots+a_nx_n)^mx_j-x_j(a_1x_1+\cdots+a_nx_n)^m$ for all $a_1,\dots,a_n\in k$ and $1\le j\le n$. The work of $\S \ref {Discriminant}$ implies that the Clifford algebra $\mathcal C_f$ is a homomorphic image of the generic Clifford algebra $\mathcal C_{m,n}$, which paves a way to understand the representations of the former through the representations of the latter. Chan-Young-Zhang further showed that (\cite[Lemma 3.8]{CYZ}) when $m=2,n=2$, the binary quadratic generic Clifford algebra belongs to the nice family of connected graded algebras that are Artin-Schelter regular (see \cite{AS}). 

In this note, we focus on the binary cubic generic Clifford algebra ($m=3,n=2$) and denote it by $\mathcal C$. In \cite{WW2012}, all Artin-Schelter regular algebras of global dimension five that are generated by two generators with three generating relations are classified. The algebra $\mathcal C$ is listed as one of the type $\mathbf{A}$ algebras there, which is also proved to be strongly noetherian, Auslander regular and Cohen-Macaulay. In this note, we further show that $\mathcal C$ is indeed a polynomial identity (PI) algebra of PI degree three (see Theorem \ref{RepC}). This and other known properties allow us to fully understand the representations of $\mathcal C$ by the well-developed representation theory of PI algebras (e.g. \cite{BG97, BG2002, BY}). As a consequence, we derive some ring-theoretic results of the binary cubic Clifford algebra $\mathcal C_f$, whose properties have been studied by Heerema \cite{Heer} and Haile \cite{Haile}. 

\begin{Ack}
This work was inspired by the poster presentation given by Charlotte Ure at Algebra Extravaganza! at Temple University in July 2017. The authors also learned from Q.-S. Wu later that the binary cubic generic Clifford algebra has appeared as one of the type $\mathbf{A}$ algebras in the classification of global dimensional five Artin-Schelter regular algebras generated by two generators with three generating relations; see \cite{WW2012}. The authors are grateful for the correspondence from them. The authors also want to thank the referee for his/her careful reading and suggestions. 
\end{Ack}

\section{Background}

\begin{note}
Let $k$ be an algebraically closed base field of characteristic not $2$ or $3$. The reader is referred (e.g.) to \cite{GW} and \cite{MR} for further background on the theory of PI rings. Recall that a ring $R$ with center $Z(R)$ is called \emph{Azumaya} over $Z(R)$ if $R$ is a finitely generated projective $Z(R)$-module and the natural map $R\otimes_{Z(R)} R^{\rm{op}}\to \rm{End}_{Z(R)}(R)$ is an isomorphism. Let $\Lambda$ be a prime noetherian affine $k$-algebra finitely generated as a module over its center $Z$. Define 
\begin{align*}
\mathcal A_\Lambda:&=\{\mathfrak m\in\, {\rm maxSpec}(Z)\, |\, \Lambda_\mathfrak m\, \text{is Azumaya over}\, Z_\mathfrak m\},\\
\mathcal S_\Lambda:&=\{\mathfrak m\in\, {\rm maxSpec}(Z)\, |\, \text{$Z_\mathfrak m$ is not regular}\}.
\end{align*}
The set $\mathcal A_\Lambda$ is called the {\it Azumaya locus} of $\Lambda$ over $Z$, and $\mathcal S_\Lambda$ is the \emph{singular locus} of $\Lambda$ (or of $Z$).
\end{note}

\begin{thm}\cite[Theorem III.1.7]{BG2002}\cite[Proposition 3.1, Lemma 3.3]{BG97}\label{PI}
Let $\Lambda$ be a prime noetherian affine $k$-algebra that is module-finite over its center $Z$.  
\begin{itemize}
\item[(a)] The maximum $k$-dimension of irreducible $\Lambda$-modules equals the PI degree of $\Lambda$.
\item[(b)] Let $S$ be an irreducible $\Lambda$-module, $P=\rm{Ann}_\Lambda(S)$, and $\mathfrak m=P\cap Z$. Thus $\dim_k(S)={\rm PI}$-${\rm deg}(\Lambda)$ if and only if $\Lambda_\mathfrak m$ is Azumaya over $Z_\mathfrak m$.
\end{itemize}
Moreover if ${\rm gldim}(\Lambda)<\infty$, then $\mathcal A_\Lambda$ is a nonempty open (and hence dense) subset of $\rm{maxSpec}(Z)\setminus \mathcal S_\Lambda$. 
\end{thm}

\begin{note}
In \cite{CYZ}, Chan-Young-Zhang defined the generic Clifford algebra and established some properties in the quadratic case. Recall that the {\it binary cubic generic Clifford algebra} $\mathcal C$ is defined as the quotient algebra of the free algebra $k\langle x,y\rangle$ subject to the relations:
\begin{align}\label{relationC}
x^3y-yx^3,\ x^2y^2+xyxy-yxyx-y^2x^2,\ xy^3-y^3x.
\end{align}
Moreover, $\mathcal C$ is one of the type $\textbf{A}$ algebras studied in \cite{WW2012} (parameters $t=-1,l_2=1$). 
%By the diamond lemma \cite{Be}, we have that $\{y^i(xy^2)^j(xy)^k(x^2y)^lx^m\,|\,i,j,k,l,m\in \mathbb N\}$ is a $k$-linear basis for $C$. 
\end{note}

\begin{thm}\cite[Theorem 5]{WW2012}
The connected graded algebra $\mathcal C$ is Artin-Schelter regular of global dimension five. Moreover, $\mathcal C$ is an Auslander regular, Cohen-Macaulay, and strongly noetherian domain. 
\end{thm}

\begin{note}
By the diamond lemma \cite{Be}, we have that $$\left\{y^i(xy^2)^j(xy)^k(x^2y)^l x^m\, |\, i,j,k,l,m\in \mathbb N\right\}$$ is a $k$-linear basis for $\mathcal C$. Consequently, the Hilbert series of $\mathcal C$ is $\displaystyle{\frac{1}{(1-t)^5(1+t)(1+t+t^2)^2}}$.
\end{note}

\begin{note}\label{Discriminant}
Let $f(u,v)=au^3+3bu^2v+3cuv^2+dv^3$ be a cubic form with $a,b,c,d\in k$. Its discriminant is given by 
\begin{align}\label{D}
D:=\frac{1}{4}\left(ad-bc\right)^2-\left(ac-b^2\right)\left(bd-c^2\right).
\end{align}
The {\it Clifford algebra} $\mathcal C_f$ of the cubic form $f$ is defined as $$\mathcal C_f:=k\langle x,y\rangle/I,$$ where $I$ is the ideal generated by the elements $(ux+vy)^3-f(u,v)$ for all $u,v\in k$. It is easy to see that 
\begin{align}\label{relationCf}
I=\left(x^3-a,\ y^3-d,\ x^2y+xyx+yx^2-3b,\ y^2x+yxy+xy^2-3c\right).
\end{align}
There is a natural surjection $\mathcal C\twoheadrightarrow \mathcal C_f$ by sending $x\mapsto x,y\mapsto y$ since the ideal generated by the defining relations in Eq.\eqref{relationC} for $\mathcal C$ is contained in $I$.\end{note}

\begin{thm}\cite[Theorem 1.1', Corollary 1.2']{Haile}\label{Clifford}
The Clifford algebra $\mathcal C_f$ is an Azumaya algebra of PI degree three if $D\neq 0$. In this case, its center is a Dedekind domain, isomorphic to the coordinate ring of the affine elliptic curve $u^2=v^3-27D$. 
\end{thm}

\section{The center and singular locus of $\mathcal C$}

\begin{note}
Consider the projective space $\mathbb P^3$ with homogenous coordinates $[z_0:z_1:z_2:z_3]$. For cubic binary forms, the discriminant projective variety is defined by $\Delta=0$, where (see Eq.\eqref{D})
\[\Delta=\frac{1}{4}\left(z_0z_3-z_1z_2\right)^2-\left(z_0z_2-z_1^2\right)\left(z_1z_3-z_2^2\right).\]
\end{note}
\begin{lem}\label{SD}
The reduced variety of the singular locus of the discriminant projective variety is given by the twisted cubic curve $v: \mathbb P^1\to \mathbb P^3$ via
\[
v: [x_0:x_1]\mapsto [x_0^3:x_0^2x_1:x_0x_1^2:x_1^3]=[z_0:z_1:z_2:z_3].
\]
\end{lem}
\begin{proof}
It is well known (e.g. \cite[Chapter I Theorem 5.1]{Hart}) that the singular locus of $\Delta=0$ is the zero locus of the polynomials
\begin{align}
\frac{\partial\Delta}{\partial z_0}&=-\frac{3}{2}z_1z_2z_3+\frac{1}{2}z_0z_3^2+z_2^3=0,\label{e1}\\
\frac{\partial\Delta}{\partial z_1}&=-\frac{3}{2}z_0z_2z_3-\frac{3}{2}z_1z_2^2+3z_1^2z_3=0,\label{e2}\\
\frac{\partial\Delta}{\partial z_2}&=-\frac{3}{2}z_0z_1z_3-\frac{3}{2}z_1^2z_2+3z_0z_2^2=0,\label{e3}\\
\frac{\partial\Delta}{\partial z_3}&=-\frac{3}{2}z_0z_1z_2+\frac{1}{2}z_0^2z_3+z_1^3=0.\label{e4}
\end{align}
If $z_0=0$, then $z_1=z_2=0$ and $z_3\neq 0$, which corresponds to the point $v([0:1])$. Now let $z_0=1$. From Eq.\eqref{e4}, we have $z_3=3z_1z_2-2z_1^3$. Substituting it into Eq.\eqref{e3}, we obtain $(z_1^2-z_2)^2=0$ which implies that $z_2=z_1^2$ and $z_3=z_1^3$. One can check that $[1:z_1:z_1^2:z_1^3]=v([1:z_1])$ satisfies Eq.\eqref{e1}-Eq.\eqref{e4}. Hence the solutions are exactly given by the twisted cubic curve. 
\end{proof}

\begin{note}
In the following, we denote by $Z$ the center of the generic Clifford algebra $\mathcal C$. We will describe the center $Z$ according to \cite{UK}. Consider the following central elements of $\mathcal C$, where $\omega\in k$ is a primitive third root of unity. 
\begin{equation}\left\{
\begin{aligned}\label{Z}
z_0&=x^3\\
z_1&=\frac{1}{3}(x^2y+xyx+yx^2)\\
z_2&=\frac{1}{3}(y^2x+yxy+xy^2)\\
z_3&=y^3\\
z_4&=(yx-\omega xy)^3-\frac{3}{2}\omega(1-\omega)x^3y^3-\frac{9}{2}(1+2\omega^2)z_1z_2\\
z_5&=(xy)^2-y^2x^2=(yx)^2-x^2y^2.
\end{aligned}\right.\end{equation}
Define the formal discriminant element $\Delta$ in $\mathcal C$ as 
\begin{align}\label{Delta}
\Delta=\frac{1}{4}\left(z_0z_3-z_1z_2\right)^2-\left(z_0z_2-z_1^2\right)\left(z_1z_3-z_2^2\right).
\end{align}
Under the natural surjection $\mathcal C\twoheadrightarrow \mathcal C_f$, it follows that $(z_0,z_1,z_2,z_3)\mapsto (a,b,c,d)$. In particular, the formal discriminant element $\Delta$ maps to the discriminant $D\in k$ of the binary cubic form $f$ under the surjection $\mathcal C\twoheadrightarrow \mathcal C_f$.
\end{note}

\begin{thm}\cite{UK}\label{T:UK}
The generic Clifford algebra $\mathcal C$ is finitely generated as a module over its center $Z$. Moreover, the center $Z$ is generated by $(z_i)_{0\le i\le 5}$ subject to one relation $z_4^2=z_5^3-27\Delta$. Consequently, ${\rm maxSpec}(Z)$ is isomorphic to the coordinate ring of a relative quasiprojective curve over the 4-dimensional affine space $\mathbb A^4=\mathrm{maxSpec}(k[z_0,z_1,z_2,z_3])$ that is elliptic over an open subset of $\mathbb A^4$.
\end{thm}

\begin{cor}\label{SC}
The reduced variety of the singular locus $\mathcal S_{\mathcal C}$ of $Z$ is an affine twisted cubic curve in $\mathbb A^4={\rm maxSpec}(k[z_0,z_1,z_2,z_3])$ defined by 
$$\mathcal S_C=\mathbb V\left(z_4,\, z_5,\, z_0z_3-z_1z_2,\, z_0z_2-z_1^2,\, z_1z_3-z_2^2\right).$$
\end{cor}
\begin{proof}
Since the center $Z$ generated by $(z_i)_{0\le i\le 5}$ is subject to one relation $R:=z_4^2-z_5^3+27\Delta$, the singular locus $\mathcal S_\mathcal C$ of $Z$ is the zero locus of the derivatives $\frac{\partial R}{\partial z_i}$ for $0\le i\le 5$, or equivalently (if ignoring multiplicity) $z_4=z_5=0$ and $\frac{\partial \Delta}{\partial z_0}=\frac{\partial \Delta}{\partial z_1}=\frac{\partial \Delta}{\partial z_2}=\frac{\partial \Delta}{\partial z_3}=0$. Thus, the result follows from Lemma \ref{SD}.
\end{proof}

\section{Irreducible representations and Azumaya locus of $\mathcal C$}
\begin{note}
It is clear that the set of isomorphism classes of one-dimensional representations of $\mathcal C$ is bijective to $\mathbb A^2$ via the one-to-one correspondence ${\rm Ann}_\mathcal C(S)=(x-a,y-b)\leftrightarrow (a,b)\in \mathbb A^2$ for any $\mathcal C$-module $S$ with $\dim_k(S)=1$. The following result is crucial to our work. 
\end{note}

\begin{lem}\label{dim2}
There are no two-dimensional irreducible representations over $\mathcal C$.
\end{lem}
\begin{proof}
Suppose $S$ is any two-dimensional irreducible representation over $\mathcal C$. This implies that there is a surjective algebra map $\varphi: \mathcal C\twoheadrightarrow {\rm End}_k(S)={\rm M}_2(k)$. One sees easily that $\varphi(x)\neq 0$; otherwise $S$ can be viewed as a module over $\mathcal C/(x)\cong k[y]$, where all irreducible representations are one-dimensional. The same argument shows that $\varphi(x)$ is not a scalar multiple of the identity matrix in ${\rm M}_2(k)$. Since $x^3$ is central in $\mathcal C$, $\varphi(x^3)=\varphi(x)^3$ is a scalar multiple of the identity matrix. Therefore by a linear transformation of $S$ and a possible rescaling of the variables $x,y$ of $\mathcal C$, we can assume that $\varphi(x)=\begin{pmatrix} 1 & 0 \\ 0 & \omega \end{pmatrix}$ where $\omega$ is a primitive third root of unity. Now write $\varphi(y)=\begin{pmatrix} a & b \\ c & d \end{pmatrix}$ for some $a,b,c,d\in k$. When we apply $\varphi$ to the relations in $\mathcal C$, we obtain
\begin{align*}
\varphi(x^2y^2+(xy)^2-(yx)^2-y^2x^2)=\begin{pmatrix}0 & 3b(a-\omega ^2d) \\ -3c(a-\omega^2d) & 0 \end{pmatrix}=0.
\end{align*}
If $a\neq \omega^2d$, then we have $b=c=0$. So $\varphi(x)$ and $\varphi(y)$ are both diagonal. This shows that $S$ is not irreducible, which is a contradiction. Hence we have $a=\omega^2 d$. Thus, $\varphi$ sends the central element $y^3$ to
$$\varphi(y^3)=\begin{pmatrix} a^3+(2+\omega)abc & b^2c \\ bc^2 &a^3+(1+2\omega)abc \end{pmatrix}.$$
Since $\varphi(y^3)$ is a scalar multiple of the identity matrix, we obtain $bc=0$ and $\varphi(y)$ is either upper or lower triangular. Again, it is a contradiction since $S$ is irreducible. 
\end{proof}
\begin{note}
We denote by ${\rm Irr}\, \mathcal C$ the isomorphism classes of all irreducible representations over $\mathcal C$. For each possible integer $n\ge 1$, ${\rm Irr}_n\, \mathcal C$ denotes the isomorphism subclasses of all $n$-dimensional irreducible representations over $\mathcal C$. Recall $\mathcal S_\mathcal C\subset {\rm maxSpec}(Z)$ is the singular locus of $\mathcal C$ and we call its complement ${\rm maxSpec}(Z)\setminus \mathcal S_\mathcal C$ the \emph{smooth locus} of $\mathcal C$. Note that there is a natural surjection 
\begin{equation}\label{SurSpec}
\chi: {\rm Irr}\, \mathcal C\twoheadrightarrow Y:={\rm maxSpec}(Z)
\end{equation} 
via $S\mapsto {\rm Ann}_\mathcal C(S)\cap Z$ for any irreducible representation $S$ over $\mathcal C$. 
\end{note}
\begin{thm}\label{RepC}
The following hold for the binary cubic generic Clifford algebra $\mathcal C$.
\begin{itemize}
\item[(a)] $\mathcal C$ is a PI algebra of PI degree three.
\item[(b)] ${\rm Irr}\, \mathcal C={\rm Irr}_1\, \mathcal C\sqcup {\rm Irr}_3\, \mathcal C$.
\item[(c)] The map ${\rm Irr}_1\, \mathcal C\twoheadrightarrow S_\mathcal C$ is three to one. 
\item[(d)] The map ${\rm Irr}_3\, \mathcal C\to  {\rm maxSpec}(Z)\setminus \mathcal S_\mathcal C$ is one to one. 
\item[(e)] The Azumaya locus of $\mathcal C$ coincides with the smooth locus of $\mathcal C$. 
\end{itemize}
\end{thm}
\begin{proof}
(a) By Theorem \ref{T:UK}, we know $\mathcal C$ is module-finite over its center $Z$ and hence is PI by \cite[Corollary 13.1.13(iii)]{MR}. Now suppose $\mathcal C$ has PI degree $n$. Regarding Eq.\eqref{SurSpec}, we know $\chi({\rm Irr}_n\, \mathcal C)$ is an open dense subset of $Y$ by Theorem \ref{PI}. Take the formal discriminant $\Delta$ of $\mathcal C$ as in Eq.\eqref{Delta}. Note that $Y\setminus \mathbb V(\Delta)$ is another open dense subset of $Y$ since $Y$ is irreducible. So there exists a maximal ideal 
$$\mathfrak m=(z_0-a,z_1-b,z_2-c,z_3-d,z_4-e,z_5-f)\in \chi({\rm Irr}_n\, \mathcal C)\setminus \mathbb V(\Delta),$$ for some $a,b,c,d,e,f\in k$. Choose any irreducible representation $S$ over $\mathcal C$ such that $\chi(S)=\mathfrak m$ and $\dim_k (S)=n$. Let $C_f$ be the Clifford algebra associated to the binary cubic form $f(u,v)=au^3+3bu^2v+3cuv^2+dv^3$. Therefore, there are algebra surjections
$$\mathcal C\twoheadrightarrow \mathcal C_f\twoheadrightarrow \mathcal C/{\mathfrak m}\, \mathcal C,$$
where the first surjection sends $\Delta$ to a nonzero discriminant $D$ in $C_f$ (see Eq.\eqref{D}). By Theorem \ref{Clifford}, $\mathcal C_f$ is an Azumaya algebra of PI degree three. So every irreducible representation over $\mathcal C_f$ is three-dimensional and the same holds for its image $\mathcal C/{\mathfrak m}\,\mathcal C$ as well. Since $\mathfrak mS=0$, we can view $S$ as an irreducible representation over $\mathcal C/{\mathfrak m}\, \mathcal C$. This implies that $\dim_k(S)=n=3$.

(b) follows from Theorem \ref{PI} and Lemma \ref{dim2} since $\mathcal C$ has PI degree three.

(c) Let $S(a,b)=\mathcal C/(x-a,y-b)$ be any one-dimensional representation over $\mathcal C$ for some $a,b\in k$. One checks by Eq.\eqref{Z} that 
$$\chi\left(S(a,b)\right)=\left(z_0-a^3,\,z_1-a^2b,\,z_2-ab^2,\,z_3-b^3,\,z_4,\,z_5\right)=:\mathfrak n.$$ 
Also we have $\chi\left(S(\omega a, \omega b)\right)=\chi\left(S(\omega^2 a, \omega^2 b)\right)=\mathfrak n$. Suppose $\chi^{-1}(\mathfrak n)=\{S_1,S_2,\dots,S_t\}$. According to \cite[Theorem 3.1(f)]{BY}, 
$$\dim_k(S_1)+\dim_k(S_2)+\cdots +\dim_k(S_t)\le {\rm PI}\text{-}{\rm deg}(\mathcal C)=3.$$ 
Hence $t\le 3$ and $\chi^{-1}(\mathfrak n)$ exactly contains three one-dimensional irreducible representations $S(a,b)$, $S(\omega a, \omega b)$ and $S(\omega^2 a, \omega^2 b)$. Moreover, $\chi$ maps ${\rm Irr}_1\, \mathcal C$ onto $\mathcal S_\mathcal C$ by Corollary \ref{SC}. The result follows.

(d) By (b), we have the surjection $\chi: \mathcal C={\rm Irr}_1\, \mathcal C\sqcup {\rm Irr}_3\, \mathcal C\twoheadrightarrow Y=\mathcal S_\mathcal C\sqcup(Y\setminus \mathcal S_\mathcal C)$.  By (c), we have an induced surjection $\chi: {\rm Irr}_3\, \mathcal C\twoheadrightarrow Y\setminus \mathcal S_\mathcal C$. Note that $\mathcal C$ has PI degree three. By the same argument as in (c), $\chi$ yields a one-to-one correspondence between ${\rm Irr}_3\, \mathcal C$ and $Y\setminus \mathcal S_\mathcal C$.

(e) is a consequence of (d) and (b) of Theorem \ref{PI}.
\end{proof}

\begin{note}
As recalled in Theorem \ref{Clifford}, Haile showed that the Clifford algebra $\mathcal C_f$ is an Azumaya algebra when the binary cubic form $f$ is nondegenerate or the discriminant $D\neq 0$ (see Eq.\eqref{D}). In the following, we give a necessary and sufficient condition for $\mathcal C_f$ to be Azumaya. It turns out that $\mathcal C_f$ can be an Azumaya algebra even when $f$ is degenerate or $D=0$. 
\end{note}

\begin{cor}
Let $k$ be a base field of characteristic different from $2$ or $3$. The Clifford algebra $\mathcal C_f$ associated to the binary cubic form $f(u,v)=au^3+3bu^2v+3cuv^2+dv^3$ is an Azumaya algebra (of {\rm PI degree} three) if and only if the point $(a,b,c,d)$ does not lie on the affine twisted cubic curve in $\mathbb A^4$ described in Lemma \ref{SD}.
\end{cor}
\begin{proof}
Without loss of generality, we can assume $k$ to be algebraically closed. Recall $Z=k[z_0,\dots,z_5]$, where $z_4^2=z_5^3-27 \Delta$, is the center of $\mathcal C$. It is easy to check that $Z$ has a  $k$-basis 
$$\{z_0^{i_0}z_1^{i_1}z_2^{i_2}z_3^{i_3}z_4^{i_4}z_5^{i_5}\,|\, 0\le i_0,i_1,i_2,i_3,i_4, 0\le i_5\le 2\}.$$ 
As a consequence, it induces a projection $\pi: {\rm maxSpec}(Z)\twoheadrightarrow \mathbb A^4$ given by the natural inclusion $k[z_0,z_1,z_2,z_3]\hookrightarrow Z$. Therefore the natural surjection $\mathcal C\twoheadrightarrow \mathcal C_f$ factors through $\mathcal C/\mathfrak m\, \mathcal C$ for some $\mathfrak m\in {\rm maxSpec}(Z)$ if and only if $\mathfrak m\in \pi^{-1}(a,b,c,d)$. 
By Theorem \ref{RepC}, we know 
\[\mathcal C/\mathfrak m\,\mathcal C\, \cong\, 
\begin{cases}
{\rm M}_3(k) & \mathfrak m\in{\rm maxSpec}(Z)\setminus S_\mathcal C\\ 
\text{a local algebra} & \mathfrak m\in S_\mathcal C
\end{cases}.
\] 
By Corollary \ref{SC}, $\pi^{-1}(a,b,c,d)\,\cap\, \mathcal S_\mathcal C\neq \emptyset$ if and only if $(a,b,c,d)$ lies in the affine twisted cubic curve in $\mathbb A^4$ given by Lemma \ref {SD}. In this case, we know $\mathcal C_f$ has irreducible representations of both dimensions one and three. Otherwise, all the irreducible representations over $\mathcal C_f$ are of dimension three. Thus the result follows by the Artin-Procesi theorem on polynomial identities (see \cite[Theorem 8.3]{Artin}).  
\end{proof}

\section{Point variety of $\mathcal C$}

\begin{note}
In noncommutative projective algebraic geometry (e.g. see \cite{ATV}), a {\it point module} $M=\bigoplus_{i\ge 0} M_i$ over $\mathcal C$ is a cyclic graded module with Hilbert series $\frac{1}{1-t}$, or namely $M$ is generated by $M_0$ and $\dim_k M_i=1$ for all $i\ge 0$. Note that every point module $M$ over $\mathcal C$ can be denoted by the following diagram
\[
\xymatrix{
M(p_0,p_1,p_2,\dots): &\underset{e_0}{\bullet}\ar@/^/[r]^-{x_0}\ar@/_/[r]_-{y_0} &\underset{e_1}{\bullet}\ar@/^/[r]^-{x_1}\ar@/_/[r]_-{y_1}& \underset{e_2}{\bullet}\ar@/^/[r]^-{x_2}\ar@/_/[r]_-{y_2}&  \underset{e_3}{\bullet}\ar@/^/[r]^-{x_3}\ar@/_/[r]_-{y_3}&  \underset{e_4}{\bullet}\ar@/^/[r]^-{x_4}\ar@/_/[r]_-{y_4} & \bullet &\cdots 
}
\]
where $M=\bigoplus_{i\ge 0} k\,e_i$ and $0\neq p_i=(x_i,y_i)\in k^2$ with the $\mathcal C$-action on $M$ given by 
\begin{align}\label{ActionM}
xe_i=x_{i}e_{i+1},\quad ye_i=y_ie_{i+1}.
\end{align} 
It is straightforward to check that two point modules $M(p_0,p_1,\dots)$ and $M(q_0,q_1,\dots)$ are isomorphic as graded $\mathcal C$ modules if and only if $p_i=q_i$ in $\mathbb P^1$ for all $i$ or $(p_0,p_1,\dots)=(q_0,q_1,\dots)$ in $\mathbb P^1\times \mathbb P^1\times \cdots$. Therefore, the isomorphism classes of all point modules over $\mathcal C$ can be parametrized by a subvariety in $\mathbb P^\infty$, which is called the {\it point variety} of $\mathcal C$. Sometimes, the point variety in $\mathbb P^\infty$ is determined by a variety $X$ in $\mathbb P^r$ for some $r\ge 1$ and an automorphism $\sigma: X\to X$, where the point variety is determined by the natural embedding $X\hookrightarrow \mathbb P^\infty$ via $x\mapsto (x,\sigma(x),\sigma^2(x),\dots)$ for any $x\in X$. In such a case for simplicity, we say the point variety is given by the pair $(X,\sigma)$. 
\end{note}

\begin{thm}\label{PS}
The point variety of $\mathcal C$ is given by the pair $(\mathbb P^1\times \mathbb P^1\times \mathbb P^1,\, {\rm id})$. 
\end{thm}
\begin{proof}
Let $M(p_0,p_1,\dots)$ be a point module. Taking any relation $f$ in Eq.\eqref{relationC} of $\mathcal C$, we have $fe_0=0$. Therefore we obtain
\begin{equation*}\left\{
\begin{aligned}
x_3x_2x_1y_0-y_3x_2x_1x_0=0\\
x_3x_2y_1y_0+x_3y_2x_1y_0-y_3x_2y_1x_0-y_3y_2x_1x_0=0\\
x_3y_2y_1y_0-y_3y_2y_1x_0=0
\end{aligned}\right.
\end{equation*} 
Consider the above equations as a system of linear equations in terms of $x_3,y_3$ and rewrite it in the matrix form.
\[
\begin{bmatrix}
x_1x_2y_0 & -x_0x_1x_2\\
x_2y_0y_1+x_1y_0y_2& -x_0x_2y_1-x_1x_0y_2\\
y_0y_1y_2 & -x_0y_1y_2
\end{bmatrix}
\begin{bmatrix}
x_3\\y_3
\end{bmatrix}
=0.
\]
It is easy to check that it has a solution $(x_3,y_3)=(x_0,y_0)$. We claim that it is the only solution up to a scalar multiple by showing that the coefficient matrix can not be identically zero. Suppose it is not true, then all the entries appearing in the $3\times 2$ matrix must be zero. Hence 
$$x_1x_2y_0=x_0x_1x_2=x_2y_0y_1+x_1y_0y_2= x_0x_2y_1+x_1x_0y_2=y_0y_1y_2=x_0y_1y_2=0.$$
One can check the equations have no solution.
%Since $x_0x_1x_2=0$, without loss of generality we can assume that $x_0=0$. Then $y_0\neq 0$. So $y_1y_2=0$. Again assume that $y_1=0$ (the case for $y_2=0$ can be treated similarly) so that $x_1\neq 0$. We further get $y_2=0$ since $x_1y_0y_2+x_2y_0y_1=x_1y_0y_2=0$ and $x_1y_0\neq 0$. So $x_2\neq 0$ and we get a contradiction since $x_1x_2y_0=0$. 

By repeating the above argument regarding the basis vector $e_i$ for all $i$, we see that $M(p_0,p_1,\dots)$ is well defined if and only if $p_i=p_{i+3}$ for all $i$. This yields our result. 
\end{proof}

\begin{note}
Let $\underline{p}=(p_0,p_1,p_2)$ be a point in $\mathbb P^1\times \mathbb P^1\times \mathbb P^1$. According to Theorem \ref{PS}, the point variety over $\mathcal C$ is given by all $M(\underline{p})$ for some $\underline{p}\in \mathbb P^1\times \mathbb P^1\times \mathbb P^1$. Denote by $\Gamma=\{(p,p,p)\,|\, p\in \mathbb P^1\}\subset \mathbb P^1\times \mathbb P^1\times \mathbb P^1$ the diagonal part. 
\end{note}

\begin{thm}
Let $M(\underline{p})$ be any point module over $\mathcal C$, and $S$ be a simple quotient of $M(\underline{p})$.
\begin{itemize}
\item[(a)] If $\underline{p}\in \Gamma$, then $\dim_k(S)=1$. 
\item[(b)] If $\underline{p}\not\in \Gamma$, then $S$ is either trivial or $\dim_k(S)=3$. 
\end{itemize}
\end{thm}
\begin{proof}
Let $M=M(p_0,p_1,p_2)=\bigoplus_{i\ge 0}k\, e_i$. We have $\mathcal C$ acts on $M$ as in \eqref{ActionM}.

(a) Suppose $p_0=p_1=p_2=(a,b)\in \mathbb P^1$. It is direct to check that the central elements $\{z_i\}_{0\le i\le 5}$ in Eq.\eqref{Z} act on $M$ as follows.
\begin{align*}
z_0e_i=a^3e_{i+3},\ z_1e_i=a^2be_{i+3},\ z_2e_i=ab^2e_{i+3},\ z_3e_i=b^3e_{i+3},\ z_4e_i=z_5e_i=0.
\end{align*} 
Hence $(z_0z_3-z_1z_2,z_0z_2-z_1^2,z_1z_3-z_2^2,z_4,z_5)\subset {\rm Ann}_\mathcal C(M)\subset {\rm Ann}_\mathcal C(S)$. By Theorem \ref{RepC} and Corollary \ref{SC}, we have $\mathbb V({\rm Ann}_\mathcal C(S)\cap Z)\in \mathcal S_\mathcal C$. Hence $\dim_k S=1$.

(b) According to Eq.\eqref{Z}, we obtain the following actions on $M$.
\begin{align*}
(z_0z_3-z_1z_2)e_i=-\frac{a}{9}e_{i+6},\ (z_0z_2-z_1^2)e_i=-\frac{b}{18}e_{i+6},\ (z_1z_3-z_2^2)e_i=-\frac{c}{18}e_{i+6},
\end{align*}
where the coefficients $a,b,c$ are given by 
\begin{align*}
a&\,=x_2y_2(x_1y_0-x_0y_1)^2+x_1y_1(x_2y_0-x_0y_2)^2+x_0y_0(x_2y_1-x_1y_2)^2\\
b&\,=x_0^2(x_2y_1-x_1y_2)^2+x_1^2(x_0y_2-x_2y_0)^2+x_2^2(x_1y_0-x_0y_1)^2\\
c&\,=y_0^2(x_2y_1-x_1y_2)^2+y_1^2(x_0y_2-x_2y_0)^2+y_2^2(x_1y_0-x_0y_1)^2.
\end{align*}
Here we write $p_i=(x_i,y_i)\in \mathbb P^1$ for all $i\ge 0$. Now suppose $\dim_kS=1$ and it is not trivial. By Theorem \ref{RepC} and Corollary \ref{SC}, we have $(z_0z_3-z_1z_2,z_0z_2-z_1^2,z_1z_3-z_2^2)\subset {\rm Ann}_\mathcal C(S)$. If any of $a,b,c$ is not zero, it implies that $e_{i+6}=0$ in $S$ for all $i$. Thus $xe_5=ye_5=0$ in $S$. Hence $e_5=0$ in $S$ otherwise $S=ke_5$ would be trivial. Repeating the argument, we obtain all $e_i=0$ in $S$. It is absurd. So $a=b=c=0$. Now set the variables 
$$X:=(x_1y_0-x_0y_1)^2,\ Y:=(x_2y_0-x_0y_2)^2,\ Z:=(x_2y_1-x_1y_2)^2.$$
We write the above conditions as a system of linear equations in terms of $X,Y,Z$ as follows.
\begin{align}\label{abc}
\begin{bmatrix}
x_2y_2 & x_1y_1 & x_0y_0\\
x_2^2    & x_1^2&x_0^2\\
y_2^2&y_1^2&y_0^2
\end{bmatrix}
\begin{bmatrix}
X\\Y\\Z
\end{bmatrix}=0.
\end{align}
We will show that Eq.\eqref{abc} has no solution whenever $\underline{p}\not\in \Gamma$, which yields that $\dim_k S>1$ if $S$ is not trivial. Hence part (b) follows from Theorem \ref{RepC}. Note that $X=0$ implies that $p_0=p_1$, $Y=0$ implies that $p_0=p_2$, and $Z=0$ implies that $p_1=p_2$. Since $\underline{p}\not\in \Gamma$, $(X,Y, Z)$ is a nonzero solution for Eq.\eqref{abc}. So
\[
0={\rm det}
\begin{pmatrix}
x_2y_2 & x_1y_1 & x_0y_0\\
x_2^2    & x_1^2&x_0^2\\
y_2^2&y_1^2&y_0^2
\end{pmatrix}:=\gamma.
\]
Direct computation shows that $\gamma^2=XYZ$. So $\mathbb V(\gamma)=\mathbb V(X)\,\cup\, \mathbb V(Y)\,\cup\, \mathbb V(Z)$. Without loss of generality, we assume that $\underline{p}\in \mathbb V(\gamma)\cap \mathbb V(X)$. Since $X=0$, we know $(x_0,y_0)=(x_1,y_1)\in \mathbb P^1$. Moreover, we can check that $Y=Z\neq 0$. But in such a case, $(X,Y,Z)$ does not satisfy Eq.\eqref{abc}. We obtain a contradiction, which proves our result. 
\end{proof}

\section{Discriminant ideals of $\mathcal C$}
\begin{note}
Discriminant ideals of PI algebras play an important role in the study of a maximal orders \cite{Reiner}, the automorphism and isomorphism problems for noncommutative algebras \cite{CPWZ}, the Zariski cancellation problem for noncommutative algebras \cite{BZ}, and the description of dimensions of irreducible representations \cite{BY}.
\end{note}

\begin{note}
Let $A$ be an algebra and $C\subseteq Z(A)$ be a central subalgebra. A \emph{trace map} on $A$ is a nonzero map $\text{tr}: A\to C$ that is cyclic ($\text{tr}(xy)=\text{tr}(yx)$ for $x,y\in A$) and $C$-linear. For a positive integer $\ell$, the \emph{$\ell$-th discriminant ideal} $D_\ell(A/C)$ and the \emph{$\ell$-th modified discriminant ideal} ${MD}_\ell(A/C)$ of $A$ over $C$ are the ideals of $C$ with generating sets
$$\left\{\text{det}([\text{tr}(y_iy_j)]_{i,j=1}^\ell)\,|\, y_1,\dots,y_\ell \in A\right\}$$ and
$$\left\{\text{det}([\text{tr}(y_iy_j')]_{i,j=1}^\ell)\,|\, y_1,y_1',\dots,y_\ell,y_\ell' \in A\right\}.$$
\end{note}

\begin{note}
Since we know the binary cubic generic Clifford algebra $\mathcal C$ is Auslander-regular, Cohen-Macaulay, and stably-free (it is connected graded), we can employ Stafford's work \cite[Theorem 2.10]{Staff} to conclude that $\mathcal C$ is a maximal order in a central simple algebra and admits the reduced trace map $\text{tr}: \mathcal C\to Z$ (e.g. \cite[Section 9]{Reiner}). The next theorem describes the zero sets of the discriminant ideals of $\mathcal C$ which has PI degree three.
\end{note}

\begin{thm} Let $\mathcal C$ be the binary cubic generic Clifford algebra of PI degree three with reduced trace map $\text{tr}: \mathcal C\to Z$. For all positive integers $\ell$, the zero sets of the $\ell$-th discriminant and $\ell$-th modified discriminant ideals of $\mathcal C$ over its center $Z$ coincide,
$\mathbb V(D_\ell(\mathcal C/Z),{\rm tr})=\mathbb V(MD_\ell(\mathcal C/Z),{\rm tr})$;
denote this set by $\mathbb V_\ell \subseteq Y:= {\rm maxSpec}(Z)$.  The following hold:
\begin{enumerate}
\item $\mathbb V_\ell=\emptyset$ for $\ell\le 3$.
\item $\mathbb V_\ell=Y^{sing}={\mathcal S}_\mathcal C$ for $4\le \ell \le 9$.
\item $\mathbb V_\ell=Y$ for $\ell>9$.
\end{enumerate}
\end{thm}
\begin{proof}
For any $\mathfrak m\in \text{maxSpec}(Z)$, denote by $\text{Irr}_\mathfrak m(\mathcal C)$ the set of isomorphism classes of irreducible representations of $\mathcal C$ with central annihilator $\mathfrak m$.
By \cite[Main Theorem (a),(e)]{BY}, we have
$$\mathbb V(D_\ell(\mathcal C/Z),\text{tr})= \mathbb V(MD_\ell(\mathcal C/Z),\text{tr})=\left\{\mathfrak m\in \text{maxSpec}(Z)\,\left |\, \sum_{V\in \text{Irr}_\mathfrak m(\mathcal C)} (\dim_k V)^2< \ell\right\}\right..$$
Hence, our result follows directly from Theorem \ref{RepC}.
\end{proof}

\end{document}